\documentclass[10pt]{amsart}

\usepackage{graphicx}
\usepackage[
text={432pt,575pt},
headheight=9pt,
centering
]{geometry}

\usepackage{lipsum}
\usepackage{indentfirst}
\usepackage[ruled]{algorithm}
\usepackage{algorithmic}

\usepackage{cite}

\usepackage{amsmath, amssymb, amsthm, enumerate, bm}
\usepackage{mathtools}
\usepackage{dsfont}
\usepackage[document]{ragged2e}

\allowdisplaybreaks 

\usepackage{mathptmx}

\usepackage{caption}
\usepackage{tabularx}
\newcolumntype{x}[1]{!{\centering\arraybackslash\vrule width #1}}
\usepackage{booktabs}
\usepackage{xcolor}
\usepackage{hyperref}
\definecolor{darkblue}{RGB}{0,0,160}
\hypersetup{
colorlinks,%
citecolor=darkblue,%
filecolor=black,%
linkcolor=darkblue,%
urlcolor=darkblue
}

\usepackage{todonotes}
\usepackage{xspace}
\newcommand{\eg}{e.\,g.,\xspace}
\newcommand{\ie}{i.\,e.,\xspace}
\newcommand{\cf}{c.\,f.,\xspace}
\newcommand{\eat}[1]{}

\newcolumntype{E}{>{\footnotesize \selectfont}l<{}} 
\newcolumntype{F}{>{\small\selectfont}l<{}}

\newcommand{\N}{\mathbb{N}}

\newcommand{\R}{\mathbb{R}}

\DeclareMathOperator{\pnt}{\raise 0.5mm \hbox{\large\bf.}}

\newtheoremstyle{thm}{}{}
     {\em}
     {}
     {\bf}
     {.}
     {0.5em}
     {\thmname{#1}\thmnumber{ #2}\thmnote{ #3}}

\newtheoremstyle{def}{}{}
     {\rm}
     {}
     {\bf}
     {.}
     {0.5em}
     {\thmname{#1}\thmnumber{ #2}\thmnote{ #3}}

\theoremstyle{thm}

\newtheorem{thm}{Theorem}[section]
\newtheorem{lem}[thm]{Lemma}
\newtheorem{cor}[thm]{Corollary}

\theoremstyle{def}

\newtheorem{rem}[thm]{Remark}
\newtheorem{exa}[thm]{Example}
\newtheorem{Def}[thm]{Definition}
\newtheorem{Alg}[thm]{Algorithm}

\makeatletter
\@namedef{subjclassname@2020}{%
  \textup{2020} Mathematics Subject Classification}
\makeatother

\let\epsilon=\varepsilon
\let\phi=\varphi
\let\kappa=\varkappa

\title[Simplicial complexes in network intrusion profiling]{
Simplicial complexes in network intrusion profiling
}

\author{Mandala von Westenholz}
\address{Osnabr\"uck University\\ Osnabr\"uck, Germany}
\email{mvonwestenho@uos.de}

\author{Martin Atzmueller}
\address{Osnabr\"uck University \& German Research Center for Artificial Intelligence (DFKI)\\ Osnabr\"uck, Germany}
\email{martin.atzmueller@uos.de}

\author{Tim R\"omer}
\address{Osnabr\"uck University\\ Osnabr\"uck, Germany}
\email{troemer@uos.de}

\subjclass[2020]{Primary: 05E45, 68R05; Secondary: 05C82, 68R10}

\keywords{simplicial complex, network intrusion profiling, simplicial patterns}

\begin{document}
\begin{abstract} 
For studying intrusion detection data we consider data points referring to individual IP addresses and their connections: We build networks associated with those data points, such that vertices in a graph are associated via the respective IP addresses, with the key property that attacked data points are part of the structure of the network. More precisely, we propose a novel approach using simplicial complexes to model the desired network and the respective intrusions in terms of simplicial attributes thus generalizing previous graph-based approaches. Adapted network centrality measures related to simplicial complexes yield so-called patterns associated to vertices, which themselves contain a set of features. These are then used to describe the attacked or the attacker vertices, respectively. Comparing this new strategy with classical concepts demonstrates the advantages of the presented approach using simplicial features for detecting and characterizing intrusions.
\end{abstract}

\maketitle

\section{Introduction} 
\label{Section: Introduction}

Network intrusion detection (\eg~\cite{Mukherjee, Sommer}) is a prominent and important research direction due to growing challenges in cyber security, \eg relating to the risk for everybody that personal data will be assaulted (\eg \cf~\cite{Rosenberg,walters2021cyber}) as well as increased cyber security measures in general~\cite{tirumala2019survey}. Hence, it is of particular interest to enhance our understanding for making sense of such situations and data, \eg by identifying and considering groups of special IP addresses like the ones that are attacked as well as the attackers themselves during a series of intrusion assaults. For that, we apply graph-based methods. In general, graph-based approaches have found various applications in computer science, mathematics, and neighbouring disciplines (see, \eg~\cite{Alippi, Bian, Cheng}).
In~\cite{Atzmüller} such an idea was used to model network intrusion detection data, in order to perform group-based analysis of critical situations, \ie specific network intrusion events. More precisely, IP addresses correspond to vertices of an associated directed graph modeled using the network intrusion data. In this graph, there is an edge between two vertices whenever there exists any kind of data exchange between the respective IP addresses. 
Given such a graph, the authors in~\cite{Atzmüller} aimed to find patterns associated to vertices, which are themselves sets of well-chosen features of the IP addresses and the respective nodes. The goal is to find common properties (\ie sets of features) of groups of vertices that best characterize groups of IP addresses of interest as good as possible, such as the ones related to attackers or attacked data points. For this purpose, graph centrality measures are used to construct specific features as properties of interesting subsets of the vertices of the graph. It turns out that graph-based patterns described in~\cite{Atzmüller} describe critical situations of IP addresses much better than non graph-based patterns related to classical constructions. 
However, there is room for improvement, since graphs only enable the detection of possible pairwise interactions between IP addresses. In contrast, higher order structures like the interaction of three or more addresses simultaneously are then not directly accessible. Motivated by the latter observation, there is the need for generalizing the graph-based approach from above in a suitable way, which still is simple enough for practical application. Simplicial complexes are fulfilling this aim, and that is why we propose these for modeling, analysis and interpretation. Recall that given a set of vertices $V$, such a simplicial complex is a set of subsets of $V$, called \emph{faces}, which is closed under taking subsets. Geometrically each of these subset corresponds to the vertices of a simplex in a given real vector space. In this work, we usually use \emph{Vietoris-Rips complexes} (see, \eg in general~\cite{ Edelsbrunner, Gilbert, Hug, Penrose} and for specific contexts, \eg~\cite{Akinwande, Reitzner, Grygierek}) where $V \subset \mathbb{R}^n$ and a set of vertices induces a face if and only if pairwise each two vertices are close enough with respect to a given metric. Given such a simplicial complex we can associate patterns of features to its vertices using the simplicial structure of the complex.
Such a simplicial approach turns out to be effective for describing and studying the structure of a network based on faces in various dimensions. For example, we might be interested in selecting a facet of the highest dimension containing a given vertex $v$ as a feature of a pattern. Then, $v$ is connected to the other vertices of this facet; in this way, we can then describe the connectivity of $v$ concerning the whole network. 
Observe, that graphs appear exactly as $1$-dimensional skeletons of simplicial complexes, only containing the faces of dimensions $0$ and $1$ -- therefore only providing a 1-dimensional perspective. In contrast, our new approach generalizes the methods in~\cite{Atzmüller} 
using simplices in given complexes of higher dimensions. We focus on centrality measures related to simplicial complexes to yield new patterns associated to vertices, which are not visible in the world of graphs.

The rest of the paper is structured as follows: Section~\ref{Section: preliminaries} presents relevant basics on simplicial complexes. Next, Section~\ref{Section: Adjacencies} provides suitable notions of adjacencies and degrees of faces of simplicial complexes. Note, that there are several ways adjacencies can be defined (see~\cite{Serrano}). In this paper, we use the variant generalizing  graph adjacencies. With these definitions so-called centralities for vertices are constructed in Section~\ref{Section: Simplicial centrality measures}, which are generalizations of graph-based centrality measures. In Section~\ref{section: Building patterns}, it is outlined how we can construct simplicial patterns with features, in particular, those ones relying on simplicial centrality measures. Next, purely graph-based patterns and simplicial ones are analyzed. For this, quality functions are introduced
which compare  measures for vertices of interest (\eg attackers in the network) in relation to all vertices, and
measures for vertices of interest inside the support of patterns in relation to its complete support. It turns out that simplicial patterns based in high dimensional structures are significantly better for several complex networks than purely graph-based patterns as studied in~\cite{Atzmüller}. We conclude the paper with an outlook in Section~\ref{Section: Outlook}, presenting a discussion with respect to further work and interesting research directions how to investigate simplicial patterns as related to a network further.
\section{Preliminaries} 
\label{Section: preliminaries}
In this section definitions related to simplicial complexes are given.
For more details, see \eg~\cite{Munkres, BobrowskiKahle, Kahle, Stanley}.

\begin{Def} \label{Def: abstract simplicial complex, k-simplex}
Let $V = \{v_1, \ldots,\, v_{m} \}$ be a finite set. A family $\Delta $ of subsets of $V$ is an \textit{abstract simplicial complex}, if $\text{for any }\sigma \in \Delta \text{ and }  \tau \subseteq \sigma\text{ it holds }  \tau \in \Delta.$
For $0 \leq k \leq m-1$ an element $\sigma \in \Delta$ with $|\sigma | =k+1$ is called an \textit{abstract k-simplex}. 
\end{Def}
For brevity, we often skip the word ``abstract'' for the corresponding objects from now on. The highest dimension $k$ of a simplex in a simplicial complex is the \emph{dimension} of the complex. Observe, that every simplicial complex can be realized as a so-called \textit{geometric simplicial complex} in $\R^n$ with sufficiently large $n$.

\begin{exa} \label{Exa: simplicial complex}
    The (abstract) simplicial complex
    \begin{align*}
       \{&\emptyset, \{x_1\}, \{x_2\}, \{x_3\}, \{x_4\} , \{x_5\}, \{x_1, x_2\}, \{x_1, x_3\}, \{x_2, x_3\},\{x_4, x_5\}, \{x_2, x_5 \},  \{x_2, x_4 \}, \{x_1,x_2,x_3 \}, \{x_4, x_5, x_2 \} \} 
    \end{align*}

    can be realized in $\R^2$. A possible realization is illustrated in the next picture below:
    \begin{center}
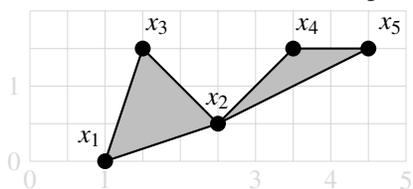

\begin{tikzpicture}  

\draw[step=0.5, gray!30] (0,0) grid (5,2);
\foreach \x in {0,...,1}
{
\node[text=gray!30, left] at (0,\x) {$\x$};
\node[text=gray!30, below] at (\x,0) {$\x$};
}
\foreach \x in {3,...,5}
{
\node[text=gray!30, below] at (\x,0) {$\x$};
}
\coordinate (A) at (1,0); 
\coordinate (C) at (1.5,1.5); 
\coordinate (F) at (4.5,1.5); 
\coordinate (E) at (3.5,0.5); 
\coordinate (D) at (2.5,0.5); 
\coordinate (L) at (14.5,0.5); 
\coordinate (I) at (3.5,1.5); 
\coordinate (J) at (4.5,0.5);
\coordinate (G) at (3.5,1.5); 
\coordinate (H) at (2.5,0.5); 

\node[font=\bfseries] (3) at (0.8,0.3) {$x_1$};
\node[font=\bfseries] (3) at (2.5,0.8) {$x_2$};
\node[font=\bfseries] (3) at (1.7,1.8) {$x_3$};
\node[font=\bfseries] (3) at (3.7,1.8) {$x_4$};
\node[font=\bfseries] (3) at (4.8,1.8) {$x_5$};

 \filldraw[fill = lightgray]
  (I) -- (D) -- (F);
 \filldraw[fill = lightgray]
  (C) -- (A) -- (D);

 \fill[black]  (A) circle [radius=3pt]; 
\fill[black]  (C) circle [radius=3pt]; 
\fill[black]  (D) circle [radius=3pt]; 	
\fill[black]  (F) circle [radius=3pt]; 
\fill[black]  (G) circle [radius=3pt]; 
  
\draw[-,>=latex, thick] (A) to (C);
\draw[-,>=latex, thick] (G) to (F);
\draw[-,>=latex, thick] (I) to (D);
\draw[-,>=latex, thick] (D) to (F);
\draw[-,>=latex, thick] (A) to (D);
\draw[-,>=latex, thick] (C) to (D);

\end{tikzpicture}
\captionof{figure}{\small{Simplicial complex }}
\end{center}
Though interesting questions in research are related to such embeddings,
there are not discussed in this manuscript except implicitly when illustrating complexes in pictures.
\end{exa}

\begin{rem} \label{Rem: constructing simplicial complexes}
For this work, it is important to note that  simplicial complexes are very good models for  describing networks. 
More precisely, objects in a network correspond to vertices. Such vertices are connected by an edge if the related objects fulfill a given relationship (which is based on the context of the network). A $k$-simplex with $k>1$ in the simplicial complex specifies whether a set of at least $k+1$ vertices fulfills some kind of pairwise relationship of interest.
\end{rem}

There are several ways to build a simplicial complex on a set of vertices. 
A key construction in various areas (see, \eg (\cite{Edelsbrunner, Gilbert, Hug, Penrose})
is given as follows:

\begin{Def} \label{Def: Vietoris-Rips}
Let $V= \{v_1, \ldots , \,v_m \} \in\mathbb{R}^n $, $r \in \mathbb{R}_{> 0}$, and a given metric $d$ on $\mathbb{R}^n$.  The \textit{Vietoris-Rips complex} $\mathcal{R}(V, \,r)$ is the simplicial complex, whose $k$-simplices $\sigma^{(k)}$ are the sets $\{v_{i_1},\ldots ,\,v_{i_{k+1}}\}\subseteq V$ with 
\[  
d(v_{i_l},v_{i_j}) \leq r \,\text{ for all }\, j, \, l \in \{1, \ldots, k+1\}. 
\]
\end{Def}

The Vietoris-Rips complex is the so-called \emph{clique complex} of the underlying graph $G$
consisting of its 0- and 1-dimensional faces. This means, that whenever there exists a $k+1$-clique in $G$, the corresponding $k$-simplex belongs to the simplicial complex.  
Note that there are many other ways to create simplicial complexes, like \textit{C\v{e}ch-complexes}. If complexes are explicitly constructed, then here usually the  Vietoris-Rips construction is used in this paper.

To construct a Vietoris-Rips complex it is required to define the distance between two data points $x_1, x_2 \in \mathbb{R}^n$ via a metric $d$. There are several reasonable options for that. Below, there are some suggestions where the objects according to Remark~\ref{Rem: constructing simplicial complexes} are IP addresses and one is interested in their interaction in the net: 

\begin{itemize}
        \item Euclidean distance $d_1(x_1, x_2) = \parallel x_2-x_1\parallel_2 $
        \item Other spatial distances like $d_2(x_1, x_2) = \parallel x_2- x_1\parallel_{\infty}$ or $d_3(x_1, x_2)=\parallel x_2- x_1 \parallel_{1}$
        \item $d_4(x_1,x_2) = \left\{
\begin{array}{ll}
\frac{1}{e} & \text{ if there is a possible data exchange of size } e \text{ between } x_1, x_2  \\
2 & \, \text{ otherwise}\\
\end{array}
\right. $  
        \item $d_5(x_1,x_2) = \left\{
\begin{array}{ll}
0 & \text{ if there is a possible data exchange between } x_1 \text{ and } x_2  \\
1 & \, \text{ otherwise}\\
\end{array}
\right. $
\end{itemize}

\begin{exa}
    The simplicial complex from the Example~\ref{Exa: simplicial complex} is a Vietoris-Rips complex, \eg with $V= \{x_1,x_2,x_3,x_4,x_5\}$, $d=d_1$, and $r = 3$. 
\end{exa}

If one is interested in relationships and in activities that are taking place at a special location, \eg in a company, it might be helpful to consider spatial distances like the Euclidean one.
There are situations where it might be more reasonable to use distances which are more closely related to possible data exchange like $d_4$ and $d_5$.

Recall that the constructions of graphs in~\cite{Atzmüller} are  based on placing an (undirected) edge between two vertices $v$ and $w$ in a finite vertex set $V$ corresponding to IP-addresses, when they interact (in one or two directions). 
These are exactly ``Vietoris-Rips-graphs'' (which are also called \emph{Gilbert graphs} or \emph{random geometric graphs}~\cite{Gilbert}) using the metric $d_5$.

\begin{exa}
Let us reconsider the simplicial complex of Example~\ref{Exa: simplicial complex}. Some exemplary distances are given as follows:
\[d_1(x_1, x_2) = \frac{\sqrt{10}}{4}, \, d_2 ( x_1, x_2) = \frac{3}{2},\, d_3 ( x_1, x_2) = 2, \, d_4( x_1, x_2 ) \text{ depends on } e, \,  d_5( x_1, x_2 ) = 0. \]
\end{exa}

Note, that one can also distinguish between incoming and outcoming data by adding directions to our constructions so far. In this work we always consider undirected situations.

\section{Adjacencies and degrees for simplicial complexes} \label{Section: Adjacencies}

Notions like adjacencies and degrees in graphs are well known  (see, \eg~\cite{Diestel}). Generalizing this to simplicial complexes and based on the work in~\cite{Serrano}, the main goal of this section is to introduce various types of \textit{adjacencies} and \textit{degrees} in such complexes.  These notions are then used, in particular, to construct later so-called \textit{features} of vertices, which use simplicial centrality measures.  

From now on a simplicial complex is always denoted by $\Delta$ and $\sigma^{(q)}$ is a simplex in $\Delta$ of dimension $q$. 
First we give the following definitions from~\cite[Section 2.1]{Serrano}, which will be important in the upcoming parts:

\begin{Def} \label{Def: upper adjacency and degree}
    Let $\emptyset \neq \sigma^{(q)}, \sigma^{(q')} \in \Delta$ be different  simplices
    and $p\in \mathbb{N}$ with $1\leq p \leq \textrm{dim}\ \Delta$. Then:
\begin{enumerate}
 \item $\sigma^{(q)}$ and $ \sigma^{(q')}$ are called $p$\textit{-upper adjacent}, denoted by \[ \sigma^{(q)} \sim_{U_p} \sigma^{(q')} \] if there exists 
 \(\textrm{a $p$-simplex } \tau ^{(p)} \in \Delta, 
 \textrm{ having both } \sigma^{(q)} \textrm{ and } \sigma^{(q')} \textrm{ as faces.}\)
    \item $\sigma^{(q)}$ and $ \sigma^{(q')}$ are called \textit{strictly $p$-upper adjacent}, denoted by $\sigma^{(q)} \sim_{U_p^*} \sigma^{(q')}\,,$ if \[\sigma^{(q)} \sim_{U_p} \sigma^{(q')} \text{ and } \sigma^{(q)} \nsim_{U_{p+1}} \sigma^{(q')} . \]
    Here $\sim_{U_{\dim \Delta+1}}$ is trivially defined.
    \item
    The $p$\textit{-upper degree} of $\sigma^{(q)}$ is
    \[ \deg_U^p( \sigma^{(q)}) =  |\{ \sigma^{(q'')}\in \Delta:  \sigma^{(q)} \sim_{U_p}  \sigma^{(q'')}\}|.\]
    \item
    The $(h,p)$\textit{-upper degree} of $\sigma^{(q)}$ is
    \[ \deg_U^{(h,p)}( \sigma^{(q)}) =  |\{ \sigma^{(q+h)}\in \Delta:  \sigma^{(q+h)} \sim_{U_{p}}  \sigma^{(q)}\}|.\]
        \item
    The \textit{strict} $(h,p)$\textit{-upper degree} of $\sigma^{(q)}$ is
    \[ \deg_U^{(h,p)^*}( \sigma^{(q)}) =  |\{ \sigma^{(q+h)}\in \Delta:  \sigma^{(q+h)} \sim_{U_{p^*}}  \sigma^{(q)}\}|.\]
    \item The \emph{maximal simplicial upper degree} of $\sigma^{(q)}$ is
    \[   \deg_U^{* } ( \sigma^{(q)}) = \sum_{h=1}^{\dim \Delta -q}  \deg_U^{(h,(q+h)) ^* } ( \sigma^{(q)}).\]
    \end{enumerate}
\end{Def}

\begin{exa}\

\begin{enumerate}
    \item     Two simplices, which are faces of a tetrahedron are $3$-upper-adjacent.
    \item Two vertices which share a common edge are $1$-upper adjacent; they are even strict $1$-upper adjacent, if they are not lying in a common triangle. 
\end{enumerate}

\end{exa}
There are also other variations of the concept of adjacency and degree, which are also introduced in~\cite[Section 2.1]{Serrano}:

\begin{Def} \label{Def: First part adjacency and degree}
Let $\emptyset \neq \sigma^{(q)}\,,\, \sigma^{(q')} \in \Delta$ denote  different simplices
    and $p\in \mathbb{N}$ with $0\leq p \leq \textrm{dim}\ \Delta-1$. Then:
\begin{enumerate}
    \item $\sigma^{(q)}$ and $ \sigma^{(q')}$ are called $p$-\textit{lower adjacent}, if there exists
    \[
    \textrm{a $p$-simplex } \tau ^{(p)} \in \Delta, \textrm{ which is a common face of both.}
    \] This is denoted by 
    \[\sigma^{(q)} \sim_{L_p} \sigma^{(q')}. \]
     \item  $\sigma^{(q)}$ and $ \sigma^{(q')}$ are called \textit{strictly} $p$\textit{-lower adjacent}, if 
    \[\sigma^{(q)} \sim_{L_p} \sigma^{(q')} \text{ and } \sigma^{(q)} \not\sim_{L_{p+1}} \sigma^{(q')}. \]
    This is denoted by 
    \[ \sigma^{(q)} \sim_{L_{p*}} \sigma^{(q')}.\]
   Here $\sim_{L_{\dim \Delta+1}}$ is trivially defined.
    \item
    The $p$\textit{-lower degree} of $\sigma^{(q)}$ is
    \[ \deg_L^p( \sigma^{(q)}) =  |\{ \sigma^{(q'')}\in \Delta:  \sigma^{(q)} \sim_{L_p}  \sigma^{(q'')}\}|.\]
   \item $\sigma^{(q)}$ and $  \sigma^{(q')}$ are called $p$\textit{-adjacent}, if they are strictly $p$-lower adjacent and not $p'$-upper adjacent for $p' = q+ q' - p$. This is denoted by 
    \[\sigma^{(q)} \sim_{A_p} \sigma^{(q')}. \]
     \item The $p$\textit{-adjaceny degree} of $\sigma^{(q)}$ is 
    \[ \deg_A^p( \sigma^{(q)}) =  |\{ \sigma^{(q'')}\in \Delta:  \sigma^{(q)} \sim_{A_p}  \sigma^{(q'')}\}|.\]
     \item $\sigma^{(q)}$ and $\sigma^{(q')}$ are called \textit{maximal} $p$\textit{-adjacent}, denoted by
        \[ \sigma^{(q)} \sim_{A_p^*}  \sigma^{(q')}\]
        if and only if
         \[ \sigma^{(q)} \sim_{A_p}  \sigma^{(q')} \text{ and }  \sigma^{(q')} \ \nsubseteq  \sigma^{(q'')} \text{ whenever } \sigma^{(q)} \sim_{A_p}  \sigma^{(q'')}.\]
    \item The \textit{maximal} $p$\textit{-adjaceny degree} of $\sigma^{(q)}$ is 
    \[ \deg_A^{p^* } ( \sigma^{(q)}) =  |\{ \sigma^{(q'')}\in \Delta:  \sigma^{(q)} \sim_{A_p^*}  \sigma^{(q'')}\}|\]

        \item The \textit{maximal} \textit{simplicial degree} of $\sigma^{(q)}$ is 
   \[ \deg^{* } ( \sigma^{(q)}) =  \deg_A^{* } ( \sigma^{(q)}) + \deg_U^{* } ( \sigma^{(q)})  ,\]
    where \[  \deg_A^{* } ( \sigma^{(q)}) = \sum_{p=0}^{q-1}  \deg_A^{p^* } ( \sigma^{(q)}).\]
\end{enumerate}
\end{Def}
Some examples for the definitions introduced so far are:

\begin{exa}
     Consider the simplicial complex from Example~\ref{Exa: simplicial complex}. The simplices $\{x_1, x_2, x_3 \}$ and $\{x_2, x_4, x_5\}$ of this complex have the following properties:
     \begin{enumerate}
         \item $x_2$ is a vertex of both simplices. Hence, they are $0$-lower-adjacent. They are strictly $0$-lower adjacent as well, as they share no simplex of higher dimension.
         \item Both simplices are not included in a common simplex of higher dimension, so they are not $p$-upper-adjacent for any $p$.
         \item  They are $0$-adjacent, since they are strictly $0$-lower adjacent and they are not $4$-upper adjacent.
         \item The $0$-adjacency degree of $\{x_1, x_2, x_3 \}$ is $3$, since besides $\{x_2, x_4, x_5\}$, also $\{x_2, x_4\}$ and $\{x_2, x_5\}$ are $0$-adjacent to $\{x_1, x_2, x_3 \}$.
         
     \end{enumerate}
 \end{exa}

In the upcoming sections, we will focus on those situations, where one simplex is always a vertex, although Definition~\ref{Def: First part adjacency and degree} is of course also usable for simplices of higher dimensions.

We observe the following in this special case:
For graphs, two vertices are called adjacent, when they both share a common edge. This graph-adjacency is covered by the definitions of adjacencies in simplicial complexes above as well. More precisely, we have:

\begin{lem} \label{Lem: graph adjancency}
Let $G=(V,E)$ be a graph and $v\in V$ be a vertex. Then: 
\begin{enumerate}
    \item  Let $w\in V$ be another vertex in $G$. Then $w$ and $v$ are graph-adjacent if and only if they are $1$-upper adjacent. 
    \item If $v\neq w\in V $, then $v, w$ are not $p$-lower adjacent and they are not $p$-adjacent for any $0\leq p \leq \dim \Delta-1$.
\end{enumerate}
 
\end{lem}

\begin{proof}
   (i) If $v$ and $w$ are graph-adjacent, then it follows immediately from the definitions, that they are $1$-upper adjacent. On the other hand, if $v$ and $w$ are $1$-upper adjacent, then they are both a face of a $1$-simplex (\ie an edge) of $G$. 
   
   (ii) The vertices $v$ and $w$ have no non-trivial face in common and are not $p$-lower adjacent for any possible $p$. This implies, that they are also not $p$-adjacent for such a $p$.
\end{proof}


\begin{lem} \label{Lem: properties vertex}
    Let $v \in \Delta$ be a vertex. Then 
    \begin{enumerate}
        \item $ \deg_L^p(v) = \deg_A^p(v) = 0$ for all $p > 0$. 
        \item $\deg_L^p(v) = 0 =\deg_A^p(v)$ for $p=0$.
        \item $\deg^{* } (v) =   \deg_U^{*}(v).$
    \end{enumerate}

\end{lem}

\begin{proof}
(i) Since a vertex has no other non-trivial simplices as subsets, the conclusion follows.

(ii) By definition $v$ is not $0$-lower adjacent to itself.  Hence, $\deg_L^p(v) =0=\deg_A^p(v)=0$.

(iii) Since $\deg_A^p(v) = 0 $, we have $\deg_A^{p^* } ( \sigma^{(q)})=0$. 

This concludes the proof.
\end{proof}

Below we are mainly interested in simplices in which a given vertex is included. 
In this context, only $p$-upper adjacencies and related degrees are of interest. 
The upper-adjacency is also reasonable to use, as it is a generalization of graph-adjacencies due to Lemma~\ref{Lem: graph adjancency}. Hence, in the following we only use $p$-upper adjacencies.

\section{Simplicial centrality measures}\label{Section: Simplicial centrality measures}
In this section, we introduce centrality measures for simplicial complexes
which are motivated by the same concept related to graphs. In the following, we focus on the so-called degree and closeness centralities. These generalizations of well-known graph degree and closeness centralities (see, \eg~\cite{Borgatti, Newman}) have all been proposed in~\cite{Serrano}. With these centrality measures we create  features for vertices of a complex as announced above in the introduction in Section~\ref{Section: Introduction}. 
To be more formal, a \textit{centrality measure} is a function 
\[c\colon I \to \R,\] where $I$ is the set, which is usually called the \emph{set of individuals}. If $I$ is a simplicial complex, then a more important simplex $\sigma \in I$, \ie one with a significantly high \emph{centrality} $c(\sigma)$, contains way more valuable information regarding the considered network than simplices with smaller centralities. There are many possibilities to define such a centrality measure. 
Hereby, the notation $c_{\sigma}$ is an abbreviation for the function value $c(\sigma)$. 

Motivated by definitions and suggestions in~\cite{Serrano} we introduce the following centrality measures for simplicial complexes. Observe, that in that work the authors propose centrality measures with images in $[0,1]$. But to be consistent with the considered centrality measures in~\cite{Atzmüller}, such normalizations are neglected in this manuscript. Then, centrality measures in~\cite{Atzmüller} are special cases of the simplicial versions given below. For simplicity, the following concepts are only defined for vertices of simplicial complexes, since in this paper in considered applications individuals are always vertices. As already mentioned in the end of Section~\ref{Section: Adjacencies}, 
only upper degrees are considered.
The first centrality measure of importance is:

\begin{Def} Let $\sigma \in \Delta $ be a vertex and $p\in \mathbb{N}_{> 0}$. Then 
\[ c_\sigma^{D^p} = \deg^p _U (\sigma )
\] is the $p$-\textit{degree centrality} of $\sigma$ and \[ c_{\sigma}^{D^*} = \deg^* (\sigma )\] is the \textit{maximal simplicial degree centrality} of $\sigma$.
\end{Def}
See~\cite[Def. 12, 13]{Serrano} for a related definition. An important property of this centrality measure is:

\begin{thm} \label{Thm: bound degree centrality}
The $p$-degree centrality of a vertex $\sigma \in \Delta$ fulfills
\begin{align} \label{Eq: bound degree centrality}
    c_\sigma^{D^p} 
\leq 
\sum_{j=1}^{p+1} 
\binom{{\rm deg}_U^{(0,1)} (\sigma)+1}{j}-1.
\end{align}

\end{thm}

\begin{proof} 
Observe that $(j-1)$-dimensional faces in $\Delta$ with $2 \leq j \leq p+1$ are  $p$-upper adjacent to $\sigma$ if they lie in a common $p$-dimensional simplex with $\sigma$. 
Given such a face all its vertices not equal to $\sigma$ have the property that they induce an edge in $\Delta$ together with $\sigma$, \ie they are $(0,1)$-adjacent to $\sigma$. Note that $\sigma$ might be a vertex of such a face or not. Hence, 
it suffices to count all possible $(j-1)$-faces induced by vertices which are either $\sigma$ or
$(0,1)$-adjacent to $\sigma$ accepting an over-count of such situations. Finally, there are at most $\deg_U^{(0,1)} (\sigma)$ many $0$-dimensional faces in $\Delta$ which are $p$-upper adjacent to $\sigma$. This concludes the proof.
\end{proof}

Using Theorem~\ref{Thm: bound degree centrality}, there is some significant computational effort needed to provide an upper bound for the $p$-degree centrality of $\sigma$. Next, we define a further degree centrality which allows alternative upper bounds:

\begin{Def} 
Let $\sigma \in \Delta $ be a vertex and $p \in \N_{> 0}$. Then 
\[ c_{\sigma}^{D^{(p,p)}} = \deg^{(p,p)} _U (\sigma )\] is the $(p,p)$-\textit{degree centrality} of $\sigma$ and \[ c_{\sigma}^{D^{(p,p)^*}} = \deg^{(p,p)^*} _U (\sigma )\] is the $(p,p)$-\textit{strict degree centrality} of $\sigma$. 

\end{Def}
See~\cite[Def. 10, 11]{Serrano} for a related definition. Note that the $(p,p)$-degree centrality of $\sigma$ is the number of $p$-dimensional faces of $\Delta$ that contain $\sigma$ and the $(p,p)$-strict degree centrality is the corresponding strict version. Since $\sigma$ is always a simplex of dimension $0$, we have
\begin{align} \label{Eq: max simpl degree and ppdegree}
    c_{\sigma}^{D^*} = \sum_{p=1}^{\dim \Delta} \deg^{(p,p)^*} _U (\sigma ).
\end{align}

\begin{rem}
A new upper bound for $c_\sigma^{D^p}$ is
\begin{align} \label{Eq: second bound degree centrality}
    c_\sigma^{D^p} 
\leq 
\deg_U^{(p,p)} (\sigma) \cdot (2^{p+1}-2)
=c_{\sigma}^{D^{(p,p)}} \cdot (2^{p+1}-2).
\end{align}
This bound follows from the fact that every non-trivial face except $\sigma$ of a $p$-simplex containing $\sigma$  is $p$-upper adjacent to $\sigma$.

The upper bound in (\ref{Eq: second bound degree centrality}) is often better than the bound in Theorem~\ref{Thm: bound degree centrality}. For example, consider a simplicial complex $\Delta$ which has exactly one tetraeder that contains $\sigma$ and $\sigma$ has additionally $t$ many $1$-upper adjacent vertices which are not pairwise connected with each other by an edge. Then, for $p=3$ the bound in (\ref{Eq: second bound degree centrality}) gives $14$ independent of $t$ where the other one is a growing function in $t$.

The disadvantage of the bound in (\ref{Eq: second bound degree centrality}) in comparison to the bound in (\ref{Eq: bound degree centrality}) is the need to find all $p$-simplices in $\Delta$ which contain $\sigma$, which is usually of more computational effort especially for $p \gg 0$ compared to determine all $(0,1)$-adjacent vertices of $\sigma$.
\end{rem}
\begin{lem} \label{Lem: degree centrality}
The $(p,p)$-degree centrality of a vertex $\sigma \in \Delta$ is anti-monotonic, i.e.
\[c_\sigma^{D^{(p,p)}}\leq c_\sigma^{D^{(q,q)}}\text{ for }1 \leq  q \leq p.\]
\end{lem}

\begin{proof}
For every $p$-simplex $\tau$ that contains $\sigma$ there exists a lower-dimensional $q$-face $\tilde{\sigma}$ of $\tau$ for each $q = p-k$ with $0 \leq k \leq p-1$ that contains $\sigma$ as well. 
\end{proof}

An upper bound for the maximal simplicial degree centrality is:

\begin{thm}
Let $\sigma \in \Delta$ be a vertex. For the maximal degree centrality of $\sigma$ we have
    \begin{align} \label{Eq: computational efford degree centrality}
   c_\sigma^{D^{*}} \leq c_\sigma^{D^{(1,1)}} \cdot\, \dim \Delta .
\end{align}
\end{thm}

\begin{proof}
The assertion follows from the following:
\begin{align*}
c_\sigma^{D^*}
&=  \sum_{p=1}^{\dim \Delta } \deg_U^{(p,p)^* } ( \sigma) 
\leq \sum_{p=1}^{\dim \Delta } \deg_U^{(p,p)} ( \sigma)  
\leq \dim \Delta \cdot \max_{p=1, \ldots, \dim \Delta} \deg_U^{(p,p) } ( \sigma) 
\\
&\leq \dim \Delta \cdot \deg_U^{(1,1) } ( \sigma) 
=  \dim {\Delta} \cdot c_\sigma^{D^{(1,1)}}.
\end{align*}
The first inequality follows directly from the definitions
and the second one is trivial to see.
The third inequality is a consequence of Lemma~\ref{Lem: degree centrality}
and the equation is just the definition of the $(p,p)$-degree centrality.
\end{proof}

Next, we discuss a typical situation where the maximal simplicial degree centrality $c_\sigma^{D^*}$ of a vertex $\sigma \in V$ is of interest. 
Consider the individual  $\sigma_2$, \ie the corresponding vertex, in the network induced from the figure shown in Example~\ref{Exa: degree centrality} below. It  has a higher maximal simplicial degree centrality $c_{\sigma_2}^{D^*} =6$ than the individual $\sigma_1$ with $c_{\sigma_1}^{D^*}= 1$ (see below). Given the context of a potential application, there might be a higher risk of $\sigma_2$ to be intruded, because the intrusion of $\sigma_2$ has a higher impact with respect to aspects measured by $c_\sigma^{D^*}$.  

Example~\ref{Exa: degree centrality} also illustrates why it might not be sufficient to consider only the $(p,p)$-degree centrality for the highest $p$ as a feature to evaluate the importance of a vertex. If we take the maximal simplicial degree centrality as well into account as a feature in the discussed situation, then we obtain an indicator, if there may be a lot of simplices which are not of maximal dimension, that include $v$ as a vertex.

\begin{exa} \label{Exa: degree centrality}
    In the following illustration, the vertices $\sigma_1, \sigma_2$ have $(p,p)$-degree centralities
    \[c_{\sigma_1}^{D^{(3,3)^*}}= 1, \, c_{\sigma_2}^{D^{(3,3)^*}}= 0 \,\text{ and }\,c_{\sigma_1}^{D^{(2,2)^*}}= 0,\, c_{\sigma_2}^{D^{(2,2)^*}}= 6.\] 
    Moreover,
    \[ c_{\sigma_1}^{D^*}= 1 < c_{\sigma_2}^{D^*} =6, \]
    although for the maximal $p$, the $(p,p)$-degree centrality of $\sigma_1$ is higher
    than the one of $\sigma_2$.
    Furthermore, the upper bounds from Equation (\ref{Eq: computational efford degree centrality}) are 
    \[ c_{\sigma_1}^{D^*}  \leq 3 \cdot c_{\sigma_1}^{D^{(1,1)}}= 9\text{ and } c_{\sigma_2}^{D^*}  \leq 2 \cdot c_{\sigma_1}^{D^{(1,1)}}=12.\]

\begin{center}

\begin{tikzpicture}

\draw[step=0.5, gray!30] (0,0) grid (12,3);

\coordinate (A) at (1,1); 
\coordinate (C) at (1.5,2.5); 
\coordinate (F) at (4.5,2.5); 
\coordinate (E) at (3.5,1.5); 

\coordinate (R) at (5,1); 
\coordinate (S) at (5.5,2.5); 
\coordinate (T) at (9.5,2.5); 
\coordinate (U) at (8.5,1.5); 

\coordinate (V) at (6,1); 
\coordinate (W) at (6.5,2.5); 
\coordinate (X) at (10.5,0.5); 
\coordinate (Y) at (7.5,0.1); 
\coordinate (K) at (6.5,0.2);

\node[font=\bfseries] (3) at (1,0.7) {$\sigma_1$};

\node[font=\bfseries] (3) at (4.8,0.7) {$\sigma_3$};
\node[font=\bfseries] (3) at (5.8,0.7) {$\sigma_2$};

 \fill[black]  (A) circle [radius=3pt]; 
\fill[black]  (C) circle [radius=3pt]; 
\fill[black]  (D) circle [radius=3pt]; 	
 
\fill[black]  (G) circle [radius=3pt]; 

\draw[-,>=latex, thick] (A) to (C);
\draw[-,>=latex, thick] (D) to (I);
\draw[-,>=latex, thick] (A) to (D);
\draw[-,>=latex, thick] (C) to (D);
\draw[-,>=latex, thick] (I) to (A);
\draw[-,>=latex, thick] (I) to (C);

 \fill[black]  (R) circle [radius=3pt]; 
\fill[black]  (S) circle [radius=3pt]; 
\fill[black]  (T) circle [radius=3pt]; 	
 
\fill[black]  (U) circle [radius=3pt]; 
 \fill[black]  (V) circle [radius=3pt]; 
\fill[black]  (W) circle [radius=3pt]; 
\fill[black]  (X) circle [radius=3pt]; 	
 
\fill[black]  (Y) circle [radius=3pt]; 
 
\fill[black]  (K) circle [radius=3pt];

\draw[-,>=latex, thick] (V) to (K);
\draw[-,>=latex, thick] (V) to (Y);
\draw[-,>=latex, thick] (Y) to (K);

\draw[-,>=latex, thick] (R) to (S);
\draw[-,>=latex, thick] (V) to (R);
\draw[-,>=latex, thick] (V) to (S);

\draw[-,>=latex, thick] (X) to (U);
\draw[-,>=latex, thick] (V) to (U);
\draw[-,>=latex, thick] (V) to (X);

\draw[-,>=latex, thick] (W) to (S);

\draw[-,>=latex, thick] (V) to (W);
\draw[-,>=latex, thick] (V) to (T);
\draw[-,>=latex, thick] (T) to (U);
\draw[-,>=latex, thick] (X) to (Y);

\end{tikzpicture}


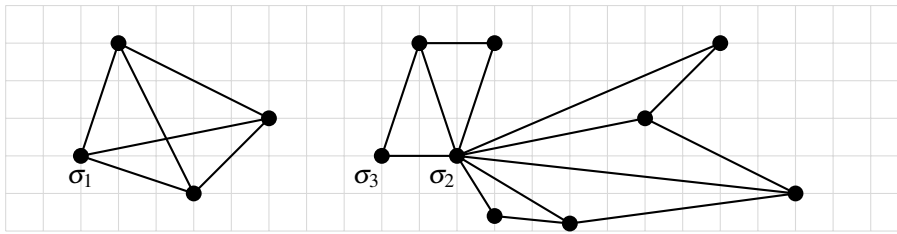
\captionof{figure}{\small{Two simplicial complexes }} \label{Abb: two simpl complexes}
\end{center}

\end{exa}

The authors of~\cite{Serrano} discuss  suggestions how to generalize \textit{closeness centralities} for vertices in graphs to new ones for simplices in simplicial complexes. Their proposals are based on the notation of $p$-adjacency. Since in this article we only use $p$-upper adjacency, we will work with the following modified versions of the definitions in~\cite[Section 3.3]{Serrano}:

\begin{Def}
Let $\sigma, \tilde{\sigma} \in \Delta$ be two vertices. $\sigma, \tilde{\sigma}$ are called $p$\textit{-connected}, if there is a path in the underlying graph of $\Delta$
starting in $\sigma$, ending in $\tilde{\sigma}$, and where two consecutive vertices of the path are $p$-upper adjacent in $\Delta$. 
\end{Def}

\begin{Def} 
Let $\sigma, \Tilde{\sigma} \in \Delta$ be two vertices and $p \in \N_{> 0}$.
\begin{enumerate}
  
    \item A $p$\textit{-walk} between $\sigma$ and $\Tilde{\sigma}$ of \emph{length} $r\in \mathbb{N}_{\geq 1}$
    is a sequence 
    \[\sigma_1 \tau_1^{(p_1)} \sigma_2 \tau_2^{(p_2)} \ldots \tau_{r-1}^{(p_{r-1})}\sigma_r \] 
    such that\ $\sigma=\sigma_1$, $\sigma_r=\Tilde{\sigma}$,   $\sigma_i$ and $\sigma_{i+1}$ are 
    $p_i$-upper adjacent with respect to $\tau_i^{(p_i)}$ for $i \in \{1, \ldots, r-1\} $ and $p = \max \{p_1, \ldots, p_r\}.$ 
\item The $p$\textit{-distance} $d_p(\sigma,\Tilde{\sigma})$ between $\sigma$ and $\Tilde{\sigma}$  is the  length of the shortest $p$-walk between $\sigma$ and $\Tilde{\sigma}$. If there are no $p$-walks, then the $p$-distance is $d_p(\sigma, \Tilde{\sigma})= \infty$.
    
\end{enumerate}
\end{Def}

The closeness centrality for vertices in graphs can now be generalized as follows:

\begin{Def} \label{Def: closeness centrality}
     Let $\sigma \in \Delta $ be a vertex and and $p \in \N_{> 0}$. The $p$\textit{-closeness centrality} of $\sigma$ is
\[ c_\sigma^{C^p} = \frac{1}{ \sum_{ \substack{\sigma \neq  \tilde{\sigma} \in \Delta\\ \text{ is a vertex}} } d_p (\sigma, \tilde{\sigma})},
\text{ where } \frac{1}{\infty} = 0  \text{ and} \]
\[c_{\sigma}^{C^{*}}= \sum_{1 \leq p \leq \dim \Delta} c_{\sigma}^{C^p}\] is the \textit{maximal closeness centrality } of $\sigma $.

\end{Def}

Similar to the degree centrality, an upper bound for the maximal closeness centrality is given as follows:

\begin{thm} \label{Thm: computational efford closeness centrality}
    \begin{align} \label{Eq: upper bound closeness centrality}
    c_{\sigma}^{C^{*}} \leq c_{\sigma}^{C^1} \cdot \dim \Delta. 
\end{align}
\end{thm}

\begin{proof}
The assertion of the Theorem follows from Definition~\ref{Def: closeness centrality} and the claim that
\[
c_{\sigma}^{C^1}  \geq c_{\sigma}^{C^p} \text{ for all } 1 \leq p \leq \dim \Delta,
\]  
if there exist a $p$-walks from $\sigma$. 

This is demonstrated as follows.
In case that two vertices $\sigma$ and $\tilde{\sigma}$ are connected via a $p$-(or lower)-dimensional simplex, then they are also contained in a $1$-face of that  simplex, \ie in an edge. Therefore, if there exists a $p$-walk between $\sigma$ and $\tilde{\sigma}$, then there exists also a $1$-walk, which has the same length as the $p$-walk, but there might also exist a $1$-walk with a smaller length. It follows for all $p$ that
  \[d_p(\sigma, \tilde{\sigma}) \geq d_1(\sigma, \tilde{\sigma}).\]
 
\end{proof}

Note that this upper bound for $c_{\sigma}^{C^{*}}$ is similar to the bound $c_\sigma^{D^{*}}$ in Equation (\ref{Eq: computational efford degree centrality}). 

\begin{exa}
Recall that the illustration in Example~\ref{Exa: degree centrality} induces two connected simplicial complexes. Then, the closeness centralities in the right simplicial complex $\Delta$ are computed for the vertex $\sigma_3$ as follows:
\[
c_3^{C^1} =  2\frac{1}{1}+6\frac{1}{2} =c_3^{C^2}.
\]
    If we add a new vertex by a new edge to $\sigma_3$ and, thus, construct a new simplicial complex $\Delta_1$, then we have in $\Delta_1$:
\[
c_3^{C^1} = 3\frac{1}{1}+ 6\frac{1}{2}.
\] 
But $c_3^{C^2}$ still equals to its old value. Hence, we see that 
\[
c_3^{C^1} \geq c_3^{C^2} \text{ in } \Delta_1.
\]
Since there is no tetraeder in $\Delta_1$, the maximal closeness centrality measure is given by 
\[c_3^{C^*}=c_3^{C^2} + c_3^{C^1}  < 2c_3^{C^1}.  \]
As a conclusion we see that the upper bound in Theorem~\ref{Thm: computational efford closeness centrality} might be strict.
\end{exa}

    For every centrality measure suggested in this section so far, there is a maximal centrality version, for which a computation for all possible $p$ is necessary. Thus, it is of interest to decide whether this is really needed. It might be more reasonable to compute the relevant information just for some $p$ which are of high interest, or to use upper bounds like in Equations (\ref{Eq: computational efford degree centrality}) and (\ref{Eq: upper bound closeness centrality}). 

Next, we give a definition of \textit{eigenvector centrality} for vertices in a simplicial complex similar to the one in~\cite[Section 3.2]{Serrano}. This shows by way of example, that other known centrality measures for vertices in graphs can be generalized as well. But first we make some preliminary considerations. 
\emph{Irreducible} matrices are square matrices $A$ that are not
conjugated to an upper block triangular matrix $B=PAP^{-1}$ with two non-trivial blocks by using (only) permutation matrices $P$.

\begin{lem} \label{lem: irreducible connected}
Let $A=(a_{ij})$ be the $m \times m$ adjacency matrix of a simple graph $G$. $A$ is irreducible if and only if $G$ is connected.
\end{lem}

\begin{proof}
$A$ being irreducible is equivalent to the fact that there do not exist any two disjoint sets $\emptyset \neq I,J \subseteq \{1, \ldots, m\}$ with $I \cup J= \{1, \ldots, m\}$, where for each pair $(i,j) \in I \times J$ it holds $a_{ij} =0$. This is equivalent to $G$ being a connected graph.
\end{proof}

The next corollary is well-known for an $m \times m$ nonnegative matrix $A$ and an associated directed graph $\mathcal{G}(A)$ (see, \eg~\cite[Theorem 1.]{Gaubert}) and can be given in an undirected setting:

\begin{cor} \label{Perron-Frobenius}
Let $G$ be a connected simple  graph with at least one edge, $A$ be its graph-adjancency matrix, and  $\lambda>0$ the largest eigenvalue of $A$.
Then $A$ has a unique normalized eigenvector $v$ with respect to $\lambda$ that has nonnegative entries. 
\end{cor}

\begin{proof}
   This follows by Lemma~\ref{lem: irreducible connected} and the version of \textit{Perron-Frobenius} in~\cite[Theorem 1.1.2 (i)]{Lemmens}, where $r(A)$ is the largest eigenvalue, since $A \neq 0$.
\end{proof}

Before one can define eigenvector centralities for a given simplicial complex, such a complex needs to have a certain property:

\begin{Def}
Let $\Delta$ be a simplicial complex and $p \in \N_{> 0}$. Then $\Delta$ is called  $p$\textit{-upper  connected}, if for each pair of two vertices $\sigma,\Tilde{\sigma} \in \Delta$
there exists a path 
\[ 
\sigma_1 \sigma_2  \cdots \sigma_r 
\]
of vertices $\sigma_i$ for $i=1, \ldots, r$ 
where
$\sigma_1=\sigma$, $\sigma_r=\Tilde{\sigma}$,
and where each two consecutive vertices in the path lie in a common $p$-face of $\Delta$. 
\end{Def}

For a given simplicial complex $\Delta$ 
with vertices $\sigma_1,\dots, \sigma_m$ 
(which are fixed for remaining part of this section) and $1\leq p \leq \dim \Delta$ 
its \emph{$p$-adjacency matrix} $A_p$
is defined via its entries
\[ a_{ij} =  \left\{
\begin{array}{ll}
0 & \, \text{if } i= j \text{ or }  \sigma_i\not\sim_{U_p} \sigma_j,\\
1 & \text{otherwise.}
\end{array}
\right.\]

\begin{thm} \label{Cor: Perron-Frobenius}
Let $\Delta$ be a $p$-upper connected simplicial complex 
with $p$-adjacency matrix $A_p \neq 0$. Let $\lambda>0$ be the largest eigenvalue of $A_p$. Then $A_p$ has a unique normalized eigenvector $c^{E^p}$ with respect to $\lambda$ that has non-negative entries.
\end{thm}

\begin{proof}
The theorem follows directly from the fact that the $p$-adjacency matrix $A_p$ of $\Delta$ is also the graph-adjacency matrix $A$ of the underlying connected graph $G$ of $\Delta$ (also called  its 1-skeleton), where additionally all edges are removed that do not lie in a $p$-simplex. Applying Corollary~\ref{Perron-Frobenius} concludes the proof. 
\end{proof}
Now we are able to introduce:
\begin{Def} \label{Def: Eigenvector centrality}
Let $p \in \N_{> 0}$ and $\Delta $ be a $p$-upper connected simplicial complex with $A_p\neq 0$
and let $c^{E^p}$ be chosen as in Theorem~\ref{Cor: Perron-Frobenius}. Set $c^{E^p}=(c_{1}^{E^p} , \ldots, c_{m}^{E^p})\in \mathbb{R}^m_{\geq 0}$.
The $p$\textit{-eigenvector centrality} is defined by mapping 
$\sigma_i$ to 
\[
c_{\sigma_i}^{E^p}=c_{i}^{E^p}.  
\]
Moreover, the \textit{maximal simplicial eigenvector centrality}
is defined as
\[
c_{\sigma_i}^{E^{*}} = \sum_{1 \leq p \leq \dim \Delta} c_{\sigma_i}^{E^p}.
\] 
\end{Def}

\begin{rem}
The $p$-\emph{eigenvector centrality} associates a score $c_{\sigma_i}^E$ to every vertex $\sigma_i \in \Delta$, which is derived in a particular way from
$ 
(A_p-\lambda I_m)x=0,  
$
where $I_m$ is the $m\times m$-\emph{identity matrix}. 
More precisely,  using the latter equation 
$c_{\sigma_i}^{E^p}$ can be computed as
\[ c_{\sigma_i}^{E^p} = \frac{1}{\lambda}\sum_
{j=1}^m a_{ij}c_{\sigma_j}^{E^p}.
\]
    Note that the $p$-eigenvector centrality from Definition~\ref{Def: Eigenvector centrality} for a $p$-upper connected simplicial complex is indeed a generalization of the eigenvector centrality for connected graphs (see, \eg~\cite{Borgatti, Newman}) for $p=1$.
\end{rem}

The next example shows that the maximal simplicial eigenvector centrality 
is in general not monotone as a function in $p$.

\begin{exa} \label{Exa: eigenvector centrality}
We claim that in the following $2$-connected simplicial complex $\Delta$ there exist two vertices $\sigma, \tilde{\sigma} \in \Delta$ with              \[c_{\sigma}^{E^1} > c_{\sigma}^{E^2} \text{ and } 
    c_{\tilde{\sigma}}^{E^2} > c_{\tilde{\sigma}}^{E^1}.\]
 \begin{center}
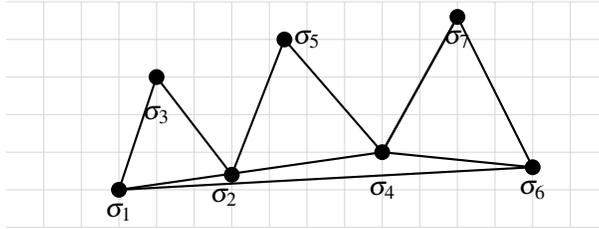

\begin{tikzpicture}  

\draw[step=0.5, gray!30] (4,0.5) grid (12,3.5);

\coordinate (A) at (5.5,1); 
\coordinate (C) at (6,2.5); 
\coordinate (F) at (9,2.5); 
\coordinate (E) at (8,1.5); 
\coordinate (D) at (7,1.2);
\coordinate (I) at (7.7,3);

\coordinate (G) at (9,1.5); 
\coordinate (H) at (10,3.3);
\coordinate (J) at (11,1.3);

\node[font=\bfseries] (3) at (5.5,0.7) {$\sigma_1$};
\node[font=\bfseries] (3) at (6.9,0.9) {$\sigma_2$};
\node[font=\bfseries] (3) at (6,2) {$\sigma_3$};
\node[font=\bfseries] (3) at (8,3) {$\sigma_5$};

\node[font=\bfseries] (3) at (9,1) {$\sigma_4$};
\node[font=\bfseries] (3) at (10,3) {$\sigma_7$};
\node[font=\bfseries] (3) at (11,1) {$\sigma_6$};

 \fill[black]  (A) circle [radius=3pt]; 
\fill[black]  (C) circle [radius=3pt]; 
\fill[black]  (D) circle [radius=3pt]; 	
 
\fill[black]  (I) circle [radius=3pt]; 
\fill[black]  (G) circle [radius=3pt]; 
\fill[black]  (H) circle [radius=3pt]; 
\fill[black]  (J) circle [radius=3pt]; 

\draw[-,>=latex, thick] (A) to (C);
\draw[-,>=latex, thick] (C) to (D);
\draw[-,>=latex, thick] (I) to (D);
\draw[-,>=latex, thick] (G) to (H);
\draw[-,>=latex, thick] (G) to (I);
\draw[-,>=latex, thick] (H) to (J);
\draw[-,>=latex, thick] (H) to (G);
\draw[-,>=latex, thick] (G) to (J);

\draw[-,>=latex, thick] (A) to (J);
\draw[-,>=latex, thick] (A) to (G);
\end{tikzpicture}
\captionof{figure}{\small{Simplicial complex $\Delta$}} \label{Abb: simplicial complex}

\end{center}

Indeed, the adjacency matrices $A_1$ for $p=1$ and $A_2$ for $p=2$ are

\[A_1 = \left( \begin{array}{rrrrrrr}
0 & 1 & 1 & 0& 0 & 0 & 0 \\
1& 0 & 1 & 1& 1 & 0 & 0  \\
1 & 1 & 0& 0& 0 & 0 & 0 \\
0& 1 & 0 & 0& 1 & 1 & 1 \\
0& 1& 0 &1& 0 &0 & 0\\
0&0&0&1&0&0&1 \\
0& 0 & 0 &1&0&1&0
\end{array} \right) \text{ and } 
A_2 = \left( \begin{array}{rrrrrrr}
0 & 1 & 1 & 0& 0 & 1 & 0 \\
1& 0 & 1 & 1& 1 & 0 & 0  \\
1 & 1 & 0& 0& 0 & 0 & 0 \\
0& 1 & 0 & 0& 1 & 1 & 1 \\
0& 1& 0 &1& 0 &0 & 0\\
1&0&0&1&0&0&1 \\
0& 0 & 0 &1&0&1&0
\end{array} \right).\]

They just differ by the entries $(1,6)$ and $(6,1)$. 
For $A_1$ let $v_1$ be the normalized eigenvector for the largest eigenvalue and observe that
\[v_1 \approx (1.33, 1.70, 1, 1.70,1.12,1.33,1).\]
Choosing similar $v_2$ for $A_2$ we obtain
\[v_2 \approx (1, 1.81, 1, 1.81,1.29,1,1).\] 
Hence, we finally find $c_{6}^{E^1} > c_{6}^{E^2} \text{ and }c_{4}^{E^2} > c_{4}^{E^1}.$
\end{exa}

\section{Building patterns} \label{section: Building patterns}

In this section, we aim to find sets of suitable features, which describe certain properties of interest of vertices in given simplicial complexes, as attackers or attacked ones. Then patterns are represented as exactly these sets of features. We give strategies and possibilities for constructing such patterns. Our definitions are motivated by the ones given in ~\cite{Atzmüller, Atzmüller2}. In comparison to~\cite{Atzmüller} the patterns will not only be equipped with graph-based features, but also with simplicial ones with respect to Vietoris-Rips complexes. In particular, centrality measures are taken into account (see Section~\ref{Section: Simplicial centrality measures}). 

To create patterns, a set of \textit{individuals} $I$ is always considered and also a set of all possible \textit{features} $F$, where a single \emph{feature} is a function 
\[f\colon I \to P_f,\] which maps an individual $i\in I$ to its property $b \in P_f$. Here, $P_f$
might be chosen differently for each feature.  Moreover, $ \mathcal{F}= \{ f_b \,|\, f \in F \text{ and } b \in P_f \}$ is the set of functions given by
\[ 
f_b \colon I \to  \{ 0,1\}
\] 
where for $i \in I$ and $f \in F$ we have $f(i) =b$ if and only if $f_b(i) =1$. 
Hence, the map $f_b$ indicates whether $i\in I$ fulfills $b$ or not. 

An \textit{attributed} simplicial complex $\Delta$ 
is a simplicial complex where each 
vertex $\sigma \in \Delta$ has an associated \textit{dataset} \[D(\sigma)= (i, F(i)),\] where $i\in I$ is chosen and $F(i)= \{ f(i) \,|\, f \in F  \}$, in which the information regarding the considered features about the corresponding individual $i$ to $\sigma$ is stored. 

From now on only attributed simplicial complexes are considered and $F$ as above is fixed. In this manuscript we build a pattern by using features exactly in the following way:
\begin{Def} 
\label{Def: pattern}
Let $f^1,\dots,f^k\in F$ and choose  $b_{j}\in P_{f^j}$ for $j=1,\dots,k$.
A set 
\[
p= \{ f^1_{b_{1}} , \ldots, f^k_{b_{k}}\} 
\]
is called a $k$-\textit{pattern}
with respect to $F$. We denote the set of $k$-patterns by $\mathcal{F}^{(k)}$.
Then, we say that $p\in \mathcal{F}^{(k)}$ is \textit{true} with respect to $i \in I$, if 
\[
p(i) = \{ f^1_{b_{1}}(i) , \ldots, f^k_{b_{k}}(i)\} = \{ 1\}.
\]
The \emph{support of $p$} is
\[ s_p = \{ i \in I \,|\, p(i) = \{1\}\}  \subseteq I,\] 
which is also called the $p$-\textit{fulfilling individuals}. Set $i_p:=|s_p|$.
\end{Def}

\begin{rem}\
\begin{enumerate}
     \item Let $\Delta$ be an attributed simplicial complex and $p$ be a compatible pattern. This pattern can also be considered as true or false with respect to a vertex $\sigma \in \Delta$ by taking the value $p(i)$, where $i$ is the first component of $D(\sigma)$.
    \item In the literature, \eg in~\cite{Atzmüller}, there exist various ways of defining patterns. For example, one could equip the patterns with more structure by using tuples instead of sets. This provides the opportunity to consider multiplicities and to order the features. In this paper pattern are always sets which is suitable for 
    the considered applications below.
\end{enumerate}
\end{rem}

For binary functions as already considered above, 
it is useful to introduce the following notation.
\begin{Def}
Let $p$ be a pattern.
A \textit{target} $t$ is a binary function
$t \colon I \to \{ 0, 1\}.$
The \textit{target share} of $t$ with respect to $p$ is  
\[
t_p :
= \frac{|\{ i \in s_p \,|\, t(i) =1\}|}{i_p}
= \frac{|\{ i \in I \,|\, p(i)=1 \text{ and } t(i) =1\}|}{i_p}.
\]
\end{Def}

For a given target, we search for patterns with a high target share. To find a useful measurement for the pattern quality and, thus, for its interestingness, we consider a \textit{quality function} (see~\cite{Atzmueller:15a,Atzmüller, Grosskreutz} for further details):

\begin{Def}
Given a target $t$, a $k$\textit{-quality function} is a real-valued function \[ q_t \colon \mathcal{F} ^ {(k)} \to \R.\] The \textit{quality} of a $k$-pattern $p$ with respect to $t$ is given by $q_t(p)$.
\end{Def}

Let $t_0 = \frac{|\{ i \in I\,|\, t(i) =1\}|}{|I|}$ be the share of the individuals that fulfill a target $t$ with respect to all individuals
and choose $a\in \mathbb{R}$ such that $(i_p)^a$ is well defined. For example, the quality of a $k$-pattern $p$ with respect to $t$ can now be determined by the following $k$-quality function:
\begin{align}
     q_t^a (p) = (i_p)^a \cdot (t_p - t_0 ).
\end{align}

Here, we target situations, for example, where $t_0$ is rather small in comparison to $t_p$ and $t_p$ itself should be much larger. The size of the pattern in terms of contained instances is then weighted by parameter $a$. Following this approach, some well-known quality functions in the literature~\cite{Atzmueller:15a} are:
\begin{itemize}
    \item the \emph{gain quality function} $ q_t^0$,
    \item the \emph{binomial test quality function}  $ q_t^{0.5}$,
    \item the \emph{Piatetsky-Shapiro quality function}  $ q_t^1$.
\end{itemize}

Then, the goal is to find patterns with a high quality value, \ie a high $q_t^a (p)$ for a pattern $p$, with a given $a$ and the target $t$ as the concept of interest.

In the remaining part of this section strategies are discussed to build and then to evaluate simplicial-based patterns. For a sample data set we investigate the problem whether simplicial features help to increase the quality of corresponding patterns in comparison to graph-based features. For this purpose, we focus on patterns in the context of specifically constructed synthetic data, which specifically features analysis options for our evaluation strategies.

For generating synthetic data in our application context, we rely on standard approaches from the field of complex networks, applying the susceptible infectious (SI) model, \eg~\cite{crepey2006epidemic,Li} for generating the non-attacker/attacker structures. Thus, at first we create synthetic data by building a simplicial complex, whose vertices are assigned as attackers or non-attackers (which are partially attacked). Using the obtained synthetic dataset, we create features using the available metrics on simplicial complexes and networks, respectively. Using these features, we can then construct patterns with respect to the targets attacker/non-attacker, for studying the impact of metrics on simplicial complexes.

\begin{rem} \label{Rem: Create simpl. complex}
The strategy for constructing the synthetic data is as follows:
\begin{enumerate}
    \item Choose an existing and real world social network which has not too many individuals. Then, these correspond to vertices, which are connected by an edge if there is interaction on the level of individuals. 
    \item Choose a number $k\ \in \mathbb{N}$ and select $k$ random vertices to be attackers.
    \item Create $s$ time periods of attacks on the given social network. Individuals who have been attacked may mutate into new attackers. 
    \item Determine a resulting network from (iii) for further investigations.
    \item Build the Vietoris-Rips complex of the underlying network from (iv) using the metric $d_5$ from Section~\ref{Section: preliminaries}.
\end{enumerate}

\end{rem}

In the following, we apply this algorithm on one specific social network, for illustrating our modeling and analysis approach using simplicial complexes.
In particular, for step (i) here we use the congress network from~\cite{Fink} which is reduced to the first $20$ individuals. Thus, all other vertices (with numbers 21--475) as well as all edges involving at least one of such vertices 
are deleted. The resulting network is illustrated in Figure~\ref{fig:congress network}.
\begin{figure}[ht]
\centering
\includegraphics[scale=.4]{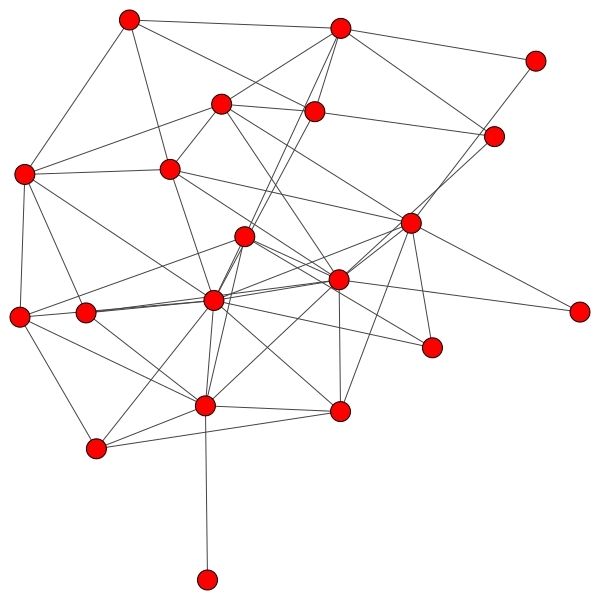}
\caption{Modified congress network}
\label{fig:congress network}
\end{figure}
For step (ii) four data points are randomly chosen with respect to the discrete uniform distribution. Then, we perform steps (iii) and (iv) for $s=1$ and use the \emph{SI model} approach described below. Finally, the Vietoris-Rips distance is chosen as $r= 1/2$ for step (v); 
see Definition~\ref{Def: Vietoris-Rips}.

For (iii) it remains to discuss shortly the SI model for modelling disease infections at a time $t$, which is  well-known in the literature (see, \eg~\cite{Li}). Let
\begin{align*}
    &S \text{ denote the number of susceptible  individuals (non-attackers), and}  \\
    &V \text{ the number of virus infectious individuals (attackers)}, 
\end{align*}
 where $S= f(t) \text{ and } V= g(t)$ for functions $f$ and $g$ with values in $\mathbb{N}$ at a time $t \geq 0$. The following assumptions are made:
\begin{enumerate}
    \item No births and deaths happen, \ie the number of total individuals $N$ is constant at any time $t$:
\begin{align} \label{Eq: constant individuals}
  f(t)+g(t) = N \text{ for all } t \geq 0,
\end{align}
    \item As suggested in~\cite[Equation 1.21]{Li}, we choose the infection reproducing rate as \[r=P\cdot \lambda/N,\] where $P$ is the probability of a contact to produce an infection and $\lambda$ is the average contact number which is also given by the average $(1,1)$-degree of the vertices in the corresponding  network.
\end{enumerate}

The change of the susceptibles is modelled by a multiple of the product of the number of the two groups (attackers and non-attackers), since this model assumes that the rate of change is proportional to the number of susceptibles and the number of infectious individuals. Moreover, $- \frac{dS}{dt} = \frac{dV}{dt}$ due to Equation (\ref{Eq: constant individuals}). 

Hence, the equation for the change of infectious and susceptible individuals is given by:
\begin{align*}
    \frac{dS}{dt} = -c V \cdot S \,  \text{ and }\,
    \frac{dV}{dt} = c V \cdot S.
\end{align*}
for  a constant $c\in \mathbb{R}$. Thus, one has
\begin{align*}
    \frac{dV}{dt} = cN V \cdot (1-V/N).
\end{align*}
Assuming that $V \ll N$ it follows $\frac{V}{N} \sim 0$. Hence,
\begin{align*}
    \frac{dV}{dt} \approx c N \cdot V = r V
\end{align*}
with the \emph{change rate} $r=cN$. Since the resulting differential equation
\begin{align} \label{Eq: Differentialgleichung}
     \frac{dV}{dt} = r V (1-\frac{V}{N})
\end{align}
 is separable, one obtains the integral equation
\begin{align*} 
  \int \frac{1}{V(1-\frac{V}{N})} \ dV = \int r \ dt,
\end{align*}
which can be solved by using partial fraction decomposition:
\begin{align*} 
 \int r \ dt= \int \frac{A}{V} + \frac{B}{1-\frac{V}{N}} \ dV=\int \frac{A(1-\frac{V}{N})}{V(1-\frac{V}{N})} + \frac{BV}{V(1-\frac{V}{N})} \ dV,
\end{align*}
which leads to the solution $A= 1$ and $B= 1/N$. So, we obtain
\begin{align*} 
 \int r dt= \int \frac{1}{N(1-\frac{V}{N})} dV+ \int\frac{1}{V} dV.
\end{align*}
Observe that by calculating the integrals we have
\begin{align*} 
 r t+c_1= \ln(V)- \ln(N-V) = \ln\frac{V}{N-V}.
\end{align*}
Finally, one receives
\begin{align*} 
 e^{r t}e^{c_1}= \frac{V}{N-V}.
\end{align*}
This equation can be transformed into 
\begin{align*} 
    V= e^{rt} c_2N- e^{rt}c_2V
\end{align*}
by defining $c_2 := e^{c_1}$. Then, $V$ can be rewritten as
\begin{align} \label{Eq: V}
    V= \frac{e^{rt}c_2N}{1+e^{rt}c_2} = \frac{e^{rt}c_2N}{e^{rt}c_2(1+e^{-rt}c_2^{-1})}= \frac{N}{1+e^{-rt}c}
\end{align}
by defining another constant $c:= c_2^{-1}$. This constant $c$ can be computed by inserting $t=0$ into the previous equation. The equation
\[ 
g(0) = \frac{N}{1+e^{-r\cdot0}c}
\]
leads to
\[ c = \frac{N}{g(0)}-1.\]
Equation (\ref{Eq: V})  implies that
\begin{align}\label{Eq: Infizierte zur Zeit t}
    g(t) = \frac{N}{1+(\frac{N}{g(0)}-1) e^{-rt}}.
\end{align}
An analogous consideration for $S$ shows that
\begin{align}\label{Eq: susceptibles zur Zeit t}
    f(t) = \frac{N}{1+(\frac{N}{g(0)}-1) e^{rt}}.
\end{align}
In the limit one has
\begin{align*}
    \lim_{t \to \infty} g(t)= \lim_{t \to \infty}\frac{N}{1+(\frac{N}{g(0)}-1)\cdot e^{-rt}} = N \text{ and }
     \lim_{t \to \infty} f(t)= \lim_{t \to \infty}\frac{N}{1+(\frac{N}{g(0)}-1)\cdot e^{rt}} = 0,
\end{align*}
since $e^{-rt} \overset{t \to \infty}{\rightarrow}0$. This means that in the limit everyone is infected.

Applying the SI model and, in particular, Equations (\ref{Eq: Infizierte zur Zeit t}) and (\ref{Eq: susceptibles zur Zeit t}) we proceed as follows for step (iii) in Remark~\ref{Rem: Create simpl. complex} using the modified congress data set:

\begin{Alg}  \label{Alg: construct data}\,
   
\begin{enumerate}
\item Choose randomly a start population of  $k=g(0)=4$ attackers. 
    \item Compute the number of virus infectious individuals $g(1)$ where the probability of a contact to produce an infection is chosen as $P=0.2$.
    \item From the set of susceptible individuals, which are connected to at least one infectious individual, choose randomly  $g(1)-g(0)$ many individuals with respect to the discrete uniform distribution. 
    \item The set of virus infectious individuals at time $1$ is then given by the infectious individual at time $0$ together with the new ones.
\end{enumerate}
  
\end{Alg}

After applying Algorithm~\ref{Alg: construct data} on the modified congress network data set, one obtains the attacker data set which is shown in Figure~\ref{fig:attackers after one time step} below (the blue vertices correspond to the attackers).

\begin{figure}[ht]
\centering
\includegraphics[scale=.4]{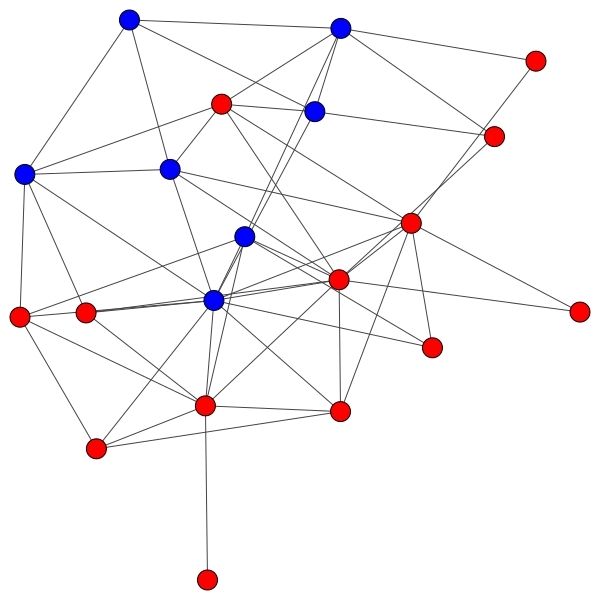}
\caption{Modified congress network with attackers after one time step of SI}
\label{fig:attackers after one time step}
\end{figure}

With the given synthetic data one is able to build simplicial patterns using measures from Section~\ref{Section: Simplicial centrality measures}. 
To keep the discussion simple, we focus on patterns of length $1$, which means that we only concentrate on one feature in each pattern. 

More precisely, the following features are discussed, where $\mathds{1}$ is the indicator-function and $k , l \in \R$. 

\begin{enumerate}

\item $\mathds{1}(c_\sigma^{D^{(p,p)}} <k) \text{ and } \mathds{1}(c_\sigma^{D^{(p,p)}} > l), $
    \item $\mathds{1}(c_\sigma^{E^p} < k) \text{ and } \mathds{1}(c_\sigma^{E^p} > l) ,$\,  
    \item $\mathds{1}(c_\sigma^{C^p} < k) \text{ and } \mathds{1}(c_\sigma^{C^p} >l).$

\end{enumerate}

We compare the chosen feature patterns of length 1 for $p=1$, which corresponds to the graph-based case that is already discussed in~\cite{Atzmüller}, and for larger $p$ with each other by computing their quality $q_{t_i}^0$ with respect to the target of finding attackers and non-attackers. This means that targets $t_1 $ and $t_2$ are considered, where $t_1$ maps an individual to $1$ if it is a non-attacker and otherwise to $0$ and $t_2$ maps an individual to $1$ if it is an attacker and otherwise to $0$. For (i) the quality values of the features 
\begin{align*}
    p_1=(c_\sigma^{D^{(1,1)}} <k_1), \, p_2=(c_\sigma^{D^{(2,2)}} <k_2)\,
    \text{ and }\, p_3=(c_\sigma^{D^{(3,3)}} <k_3)
\end{align*}
are analyzed with respect to $t_1$. 
Moreover, the quality values of the features 
\[p_3=(c_\sigma^{D^{(1,1)}} >l_1), \, p_4=(c_\sigma^{D^{(2,2)}} >l_2)\, \text{ and } \, p_5=(c_\sigma^{D^{(3,3)}} >l_3)\] 
are analyzed with respect to $t_2$. One could also construct length-one patterns using features $c_\sigma^{D^{(p,p)}}$ for $p\geq 4$, but in this example it suffices to consider $p < 4$. As illustrated in Figure~\ref{fig:attackers after one time step} no simplices exist in the network
with dimensions higher than $3$  (see Table~\ref{Tabl:(p,p)-degree centralities}). Here we choose 
\[k_1= 4, k_2= 2 \text{ and } k_3=1 \text{ as well as } l_1= 5, l_2= 14 \text{ and } l_3=7.
\]
These choices turn out to deliver the best quality values $q_{t_1}^0$ (resp. $q_{t_2}^0$) for $c_\sigma^{D^{(i,i)}} > l_i$ (resp. $c_\sigma^{D^{(i,i)}} <k_i$) compared to other possible numbers. 

The features as well as the property of being an attacker or a non-attacker are shown in Table~\ref{Tabl:(p,p)-degree centralities}.

\begin{table}[ht]
\begin{tabular}[h]{l|l|l|l|l|l|l|l|l|l|l|l|l}
\scriptsize{vertex}  & \scriptsize{$c^{D^{(1,1)}}$} & \scriptsize{$c^{D^{(2,2)}}$} & \scriptsize{$c^{D^{(3,3)}}$} &\scriptsize{$c^{D^{(4,4)}}$} & \scriptsize{attacker?}&\scriptsize{non-attacker?} &\footnotesize{$p_1$ }& \scriptsize{$p_2$ }& \footnotesize{$p_3$}& \scriptsize{$p_4$} &\footnotesize{$p_5$} &\scriptsize{ $p_6$}\\
\hline
0& 6 & 5 & 1& 0 & 0 & 1 &0 &0 &0 &1 &0 & 0 \\
 1&3&2&0&0&0&1&1&0&1&0&0&0 \\
2& 3&1&0&0&0&1&1&1&1&0&0&0 \\
3& 5&6&2&0&0&1&0&0&0&0&0&0 \\
4& 6&5&0&0&1&0&0&0&1&1&0&0 \\
5& 2&1&0&0&0&1&1&1&1&0&0&0 \\
6& 1& 0& 0& 0& 0& 1&1&1&1&0&0&0 \\
7& 5&7&3&0&0&1&0&0&0&0&0 &0 \\
8& 8&12&6&0&0&1&0&0&0&1&0&0\\
9& 6&3&0&0&1&0&0&0&1&1&0&0 \\
10& 2&0&0&0&0&1&1&1&1&0&0&0 \\
11& 8&9&3&0&0&1&0&0&0&1&0&0 \\
12& 6&8&2&0&1&0&0&0&0&1&0&0 \\
13& 10&14&6&0&0&1&0&0&0&1&0&0 \\
14& 4&2&0&0&1&0&0&0&1&0&0&0 \\
15& 6&6&2&0&1&0&0&0&0 &1&0&0 \\
16& 4&4&1&0&0&1&0&0&0&0&0&0 \\
17& 12&21&8&0&1&0&0&0&0&1&1&1 \\
18& 5&3&0&0&1&0&0&0&1&0&0&0
\\
19& 4&5&2&0&0&1&0&0&0&0&0&0
\end{tabular}
\caption{$(p,p)$-degree centralities}
\label{Tabl:(p,p)-degree centralities}
\end{table}


More precisely, the considered patterns have the following quality values: 
\begin{align*}
q^0_{t_1}(p_1)&= i_p^0 \cdot (t_p-t_0) = 5^0 \cdot (5/5 - 13/20) = 7/20, \\
    q^0_{t_1}(p_2)&= 4^0 \cdot (4/4 - 13/20) =7/20, \\
    q^0_{t_1}(p_3)&=  9^0 \cdot (5/9 - 13/20) = 17/180, \\
    q^0_{t_2}(p_4)&=  9^0 \cdot (5/9 - 7/20) = 37/180, \\
    q^0_{t_2}(p_5)&=  1^0 \cdot (1 - 7/20) =13/20, \\
    q^0_{t_2}(p_6)&= 1^0 \cdot (1 - 7/20) = 13/20.
\end{align*}

Thus, the $(1,1)$-degree centrality is as good as the $(2,2)$-degree centrality and they have both a higher quality than the $(3,3)$-degree centrality with respect to the pattern $t_1$ and the chosen quality function. The $(2,2)$-degree centrality is as good as the $(3,3)$-degree centrality and they have both a higher quality than the $(1,1)$-degree centrality with respect to the pattern $t_2$ and $q^0_{t_2}$. 
Hence, $p=2$ is the best choice for considering a simplicial degree centrality feature pattern of only length $1$ given the described setup.

The features in (ii) are considered in Table ~\ref{Tabl:(p,p)-eigenvector centralities 1}. Here we analyze the following patterns with respect to $t_1$:
\[p_1=(c_\sigma^{E^1} <k_1), \, p_2=(c_\sigma^{E^2} <k_2)\,\text{ and } \,p_3=(c_\sigma^{E^3} <k_3)\] 
with $k_1= 0.5, k_2= 0.2 \text{ and } k_3=2.$
Moreover, we investigate the following patterns with respect to $t_2$: 
\[p_4=(c_\sigma^{E^1} >l_1), \, p_5=(c_\sigma^{E^2} >l_2)\,\text{ and } \,p_6=(c_\sigma^{E^3} >l_3)\] 
with $l_1= 2.5, l_2= 2.5 \text{ and } l_3=2.5.$
Here the chosen $k_1, k_2, k_3$ and $l_1, l_2, l_3$ deliver the best quality values with respect to $q_{t_1}^0$ and $q_{t_2}^0$. These numbers are not unique with this property. \eg for $l_3$ one can also choose $2.4$. It is not helpful to consider eigenvector centralities for  $p\geq 4$ for the modified congress network, since then the eigenvector centralities are $0$ for all vertices
(see Table~\ref{Tabl:(p,p)-eigenvector centralities 1} for examples). Thus, it is not possible to use such $p\geq 4$ to create patterns for distinguishing attackers and non-attackers. 

\begin{table}[ht]
\begin{center}   
\begin{tabular}[h]{l|l|l|l|l|l|l|l|l|l|l|l|l}
\scriptsize{vertex } & \scriptsize{$c^{(E^{1})}$} & \scriptsize{$c^{(E^{2})}$} & \scriptsize{$c^{(E^{3})}$} & \scriptsize{$c^{(E^{4})}$} & \scriptsize{attacker?} &\scriptsize{non-attacker?} &\scriptsize{$p_1$} & \scriptsize{$p_2$} & \scriptsize{$p_3$}& \scriptsize{$p_4$} &\scriptsize{$p_5$ }& \scriptsize{$p_6$}\\
\hline
0&1.421& 1.279& 0.913 & 0.0& 0 & 1 & 0 &0 &1 &0 &1 &0  \\
 1&0.962&0.973&0.0&0.0&0&1&0&0&1&0&0&0 \\
2& 0.663&0.122&0.0&0.0&0&1&0&1&1&0&0&0 \\
3& 1.384&1.430&1.338&0.0&0&1&0&0&1&0&0&0 \\
4& 1.442&1.455&0.0&0.0&1&0&0&0&1&0&0&0 \\
5& 0.648&0.661&0.0&0.0&0&1&0&0&1&0&0&0 \\
6& 0.298& 0.0& 0.0& 0.0& 0& 1&1&1&1&0&0&0 \\
7& 1.564&1.617&1.760&0.0&0&1&0&0&1&0&0 &0 \\
8& 1.935&1.962&2.134&0.0&0&1&0&0&0&0&0&0\\
9& 0.908&0.384&0.0&0.0&1&0&0&0&1&0&0&0 \\
10& 0.421&0.0&0.0&0.0&0&1&1&1&1&0&0&0 \\
11& 1.828&1.803&1.564&0.0&0&1&0&0&1&0&0&0 \\
12& 1.643&1.640&1.299&0.0&1&0&0&0&1&0&0&0 \\
13& 2.380&2.336&2.293&0.0&0&1&0&0&0&0&0&0 \\
14& 0.772&0.617&0.0&0.0&1&0&0&0&1&0&0&0 \\
15& 1.601&1.514&1.475&0.0&1&0&0&0&1 &0&0&0 \\
16& 1.321&1.362&1.238&0.0&0&1&0&0&1&0&0&0 \\
17& 2.816&2.772&2.565&0.0&1&0&0&0&0&1&1&1 \\
18& 1.014&0.384&0.0&0.0&1&0&0&0&1&0&0&0 
\\
19& 1.187&1.244&1.381&0.0&0&1&0&0&1&0&0&0
\end{tabular}
\end{center}
\caption{$p$-eigenvector centralities}
\label{Tabl:(p,p)-eigenvector centralities 1}
\end{table}

A computation yields: 
\begin{align*}
q^0_{t_1}(p_1)&= 2^0 \cdot (2/2 - 13/20) = 7/20, \\
    q^0_{t_1}(p_2)&=  3^0 \cdot (3/3 - 13/20) =7/20, \\
    q^0_{t_1}(p_3)&= 17^0 \cdot (12/17 - 13/20) = 19/340, \\
    q^0_{t_2}(p_4)&= 1^0 \cdot (1/1 - 7/20) = 13/20, \\
    q^0_{t_2}(p_5)&=  1^0 \cdot (1/1 - 7/20) =13/20, \\
    q^0_{t_2}(p_6)&= 1^0 \cdot (1/1 - 7/20) = 13/20.
\end{align*}

Thus, the $1$-eigenvector centrality is as good as the $2$-eigenvector centrality and they have both a higher quality than the $3$-eigenvector centrality with respect to the pattern $t_1$ and the chosen quality function. 
The $1$-eigenvector centrality is as useful as the $2$-eigenvector centrality and the $3$-eigenvector centrality with respect to the pattern $t_2$. 
So in total the eigenvector-centralities for $p=1$ and $p=2$ are the best choices regarding the simplicial-eigenvector-feature-patterns of length $1$. 

For (iii) only features for $c_\sigma^{C^{1}}$ are considered, since the closeness centrality measure is always $0$ for other $p$, see Table \ref{Tabl:closeness}. 
 The best quality for characterizing non-attackers (target $t_1$) is achieved with the pattern
\begin{align*}
    p_1=(c_\sigma^{C^{1}} <0.024).
\end{align*} 
Moreover, the feature
\begin{align*}
    p_2=(c_\sigma^{C^{1}} >0.027)
\end{align*}
has the highest quality with respect to target $t_2$, \ie for characterizing attackers. As in (ii) the chosen numbers are not unique with respect to this property.

\begin{table}[ht]
\begin{center}   
\begin{tabular}[h]{l|l|l|l|l|l|l}
\scriptsize{vertex} & \scriptsize{attacker?} &\scriptsize{non-attacker?} & \scriptsize{$c^{C^{1}}$} &\scriptsize{$p_1$} & \scriptsize{$p_2$} \\
\hline
0& 0 & 1 & 0.029&1 & 0 \\
1& 1 & 0 & 0.026&0 & 0 \\
2& 1 & 0 & 0.025&0 & 0 \\
3& 0& 1 & 0.028&1 & 0 \\
4& 0& 1 & 0.028&1 & 0 \\
5& 1& 0 & 0.023&0 & 1 \\
6& 1& 0 & 0.02&0 & 1 \\
7& 1& 0 & 0.029&1 & 0 \\
8& 1& 0 & 0.031&1 & 0 \\
9& 1& 0 & 0.027&1 & 0 \\
10& 1& 0 & 0.023&0 & 1 \\
11& 1& 0 & 0.032&1 & 0 \\
12& 0& 1 & 0.030&1 & 0 \\
13& 1& 0 & 0.036&1 & 0 \\
14& 1& 0 & 0.024&0 & 0 \\
15& 0& 1 & 0.031&1 & 0 \\
16& 1& 0 & 0.028&1 & 0 \\
17& 1& 0 & 0.039&1 & 0 \\
18& 0& 1 & 0.029&1 & 0 
\\
19& 1& 0 & 0.025 &1 & 0
\end{tabular}
\end{center}
\caption{$p$-closeness centralities}
\label{Tabl:closeness}
\end{table}


More precisely, $p_1$ and $p_2$ have the following quality values: 
\begin{align*}
    q^0_{t_1}(p_1)&= i_p \cdot (t_p-t_0) = 3^0 \cdot (3/3 - 13/20) =7/20, \\
    q^0_{t_2}(p_2)&= i_p \cdot (t_p-t_0) = 13^0 \cdot (6/13 - 7/20) = 29/260. 
\end{align*}

The pattern $p_1$ has a rather high quality and is thus useful for further investigations. The constructed feature pattern with respect to the attacker target $t_2$ is not that beneficial, but there may be additional other possibilities for using this feature (which are not considered here in this manuscript) like combining it with other features in longer patterns. 
Observe that in contrast to (i) and (ii)
in (iii) for the modified congress network it was not possible to construct helpful simplicial features of complexes of higher dimensions than 1, but only graph-based features using the closeness centrality. In consequence it would be reasonable to consider here in the future other networks, where there might be a denser structure for higher dimensional simplices to find vertices that have non-zero $p$-closeness centrality for $p > 1$.
\begin{rem}
To avoid the synthetic data used in this section one has to record non-trivial real intrusion data, which contains not only the attacks, but also the interaction in-between groups of attackers and of attacked IP-addresses. 
This is left as an interesting research problem for the future.
\end{rem}
\section{Outlook} \label{Section: Outlook}
In this work, we applied simplicial complexes for detecting higher-order patterns, as simplicial-based patterns, in the context of network intrusion detection settings.
In Section \ref{section: Building patterns} we considered a synthetic data set and assessed how good the considered simplicial feature patterns worked for it. In the following, we outline perspectives on further eligible strategies to evaluate the usage of simplicial-based patterns for describing targets, \eg to differentiate attacker and non-attackers in networks based on real data. There are several options to continue the investigations from our work:
\begin{enumerate}
    \item Using various (further) datasets to analyze the utility of simplicial feature patterns. For this it is reasonable to vary synthetic or real, large or small data, as well as data from different application domains.
    \item Constructing patterns that use more simplicial complex features at once. It is reasonable to consider patterns that have length $>1$ to study further centrality measures which were not yet mentioned in this manuscript.
    \item Equipping simplicial patterns with further features which are not originating from simplicial centrality measures. It is also possible to construct further features from one given centrality measure by using new conditions. Here are some examples for the latter suggestion to name just a few:
    \begin{itemize}
        \item Maximal dimension of a vertex-$v$-including facet is higher than $k$: \[\mathds{1}(\max_{v \in \sigma, \sigma \in \Delta}{\dim(\sigma) > k}  ).\]
     \item Importance of a vertex-$v$-including facet in the simplicial complex regarding the centrality measure $c$ is higher than $k$: \[\mathds{1}(\max_{v \in \sigma, \sigma \in \Delta}c_{\sigma} > k  ).\] 
     \item Maximal eigenvector-centrality of a vertex $\sigma$ is higher than $k$:
     \[\mathds{1}(c_{\sigma}^{C^*}> k  ).\]
    \end{itemize}
    \item Using multiple types of simplicial complexes like \v{C}ech complex and Vietoris-Rips complex and combine features depending on several simplicial complexes in one pattern. Consider different metrics and choose the simplicial complexes that lead to the best results.
\end{enumerate}

\bibliography{simplicial-complex-paper-bibliography}{}

\begin{thebibliography}{10}

\bibitem{Akinwande}
{\sc G.~Akinwande and M.~Reitzner}, {\em Multivariate central limit theorems for random simplicial complexes}, Advances in Applied Mathematics, 121 (2020), p.~102076.

\bibitem{Atzmueller:15a}
{\sc M.~Atzmueller}, {\em {Subgroup Discovery}}, WIREs Data Mining and Knowledge Discovery, 5 (2015), pp.~35--49.

\bibitem{Atzmüller2}
{\sc M.~Atzmueller, H.~Soldano, G.~Santini, and D.~Bouthinon}, {\em Minerlsd: efficient mining of local patterns on attributed networks}, Applied Network Science, 4 (2019), pp.~1--33.

\bibitem{Atzmüller}
{\sc M.~Atzmueller, S.~Sylvester, and R.~Kanawati}, {\em Exploratory and explanation-aware network intrusion profiling using subgroup discovery and complex network analysis}, in Proceedings of the 2023 European Interdisciplinary Cybersecurity Conference, 2023, pp.~153--158.

\bibitem{BobrowskiKahle}
{\sc O.~Bobrowski and M.~Kahle}, {\em Topology of random geometric complexes: a survey}, Journal of applied and Computational Topology, 1 (2018), pp.~331--364.

\bibitem{Borgatti}
{\sc S.~P. Borgatti and M.~G. Everett}, {\em A graph-theoretic perspective on centrality}, Social networks, 28 (2006), pp.~466--484.

\bibitem{crepey2006epidemic}
{\sc P.~Crepey, F.~P. Alvarez, and M.~Barth{\'e}lemy}, {\em Epidemic variability in complex networks}, Physical Review E, 73 (2006), p.~046131.

\bibitem{Diestel}
{\sc R.~Diestel}, {\em Graphentheorie}, Berlin: Springer Spektrum., 5~ed., 2017.

\bibitem{Edelsbrunner}
{\sc H.~Edelsbrunner and J.~L. Harer}, {\em Computational topology: an introduction}, American Mathematical Society, 2022.

\bibitem{Fink}
{\sc C.~G. Fink, K.~Fullin, G.~Gutierrez, N.~Omodt, S.~Zinnecker, G.~Sprint, and S.~McCulloch}, {\em A centrality measure for quantifying spread on weighted, directed networks}, Physica A: Statistical Mechanics and its Applications, 626 (2023), p.~129083.

\bibitem{Gaubert}
{\sc S.~Gaubert and J.~Gunawardena}, {\em The perron-frobenius theorem for homogeneous, monotone functions}, Transactions of the American Mathematical Society, 356 (2004), pp.~4931--4950.

\bibitem{Gilbert}
{\sc E.~N. Gilbert}, {\em Random plane networks}, Journal of the society for industrial and applied mathematics, 9 (1961), pp.~533--543.

\bibitem{Grosskreutz}
{\sc H.~Grosskreutz, S.~R{\"u}ping, and S.~Wrobel}, {\em Tight optimistic estimates for fast subgroup discovery}, in Joint European conference on machine learning and knowledge discovery in databases, Springer, 2008, pp.~440--456.

\bibitem{Grygierek}
{\sc J.~Grygierek, M.~Juhnke-Kubitzke, M.~Reitzner, T.~R{\"o}mer, and O.~R{\"o}ndigs}, {\em Gigantic random simplicial complexes}, Homology Homotopy Appl., 22 (2020), pp.~297--318.

\bibitem{Hug}
{\sc D.~Hug and M.~Reitzner}, {\em Introduction to stochastic geometry}, Stochastic Analysis for Poisson Point Processes: Malliavin Calculus, Wiener-It{\^o} Chaos Expansions and Stochastic Geometry,  (2016), pp.~145--184.

\bibitem{Kahle}
{\sc M.~Kahle and E.~Meckes}, {\em Limit theorems for betti numbers of random simplicial complexes}, Homology, Homotopy and Applications, 15 (2013), pp.~343--374.

\bibitem{Bian}
{\sc R.~Kosk, R.~Southern, L.~You, S.~Bian, W.~Kokke, and G.~Maguire}, {\em Deep spectral meshes: Multi-frequency facial mesh processing with graph neural networks}, arXiv preprint arXiv:2402.10365,  (2024).

\bibitem{Lemmens}
{\sc B.~Lemmens and R.~Nussbaum}, {\em Nonlinear Perron-Frobenius Theory}, vol.~189, Cambridge University Press, 2012.

\bibitem{Li}
{\sc M.~Y. Li}, {\em An introduction to mathematical modeling of infectious diseases}, vol.~2, Springer, 2018.

\bibitem{Alippi}
{\sc I.~Marisca, C.~Alippi, and F.~M. Bianchi}, {\em Graph-based forecasting with missing data through spatiotemporal downsampling}, arXiv preprint arXiv:2402.10634,  (2024).

\bibitem{Mukherjee}
{\sc B.~Mukherjee, L.~T. Heberlein, and K.~N. Levitt}, {\em Network intrusion detection}, IEEE network, 8 (1994), pp.~26--41.

\bibitem{Munkres}
{\sc J.~R. Munkres}, {\em Elements of algebraic topology}, CRC press, 2018.

\bibitem{Newman}
{\sc M.~Newman}, {\em Networks}, Oxford university press, 2018.

\bibitem{Penrose}
{\sc M.~Penrose}, {\em Random geometric graphs}, vol.~5, Oxford University Press, 2003.

\bibitem{Reitzner}
{\sc M.~Reitzner, T.~R{\"o}mer, and M.~von Westenholz}, {\em Covariance matrices of length power functionals of random geometric graphs--an asymptotic analysis}, Linear Algebra and its Applications, 691 (2024), pp.~151--181.

\bibitem{Rosenberg}
{\sc J.~Rosenberg}, {\em Security in embedded systems}, Rugged Embedded Systems, 3 (2017).

\bibitem{Serrano}
{\sc D.~H. Serrano and D.~S. G{\'o}mez}, {\em Centrality measures in simplicial complexes: Applications of topological data analysis to network science}, Applied Mathematics and Computation, 382 (2020), p.~125331.

\bibitem{Sommer}
{\sc R.~Sommer and V.~Paxson}, {\em Outside the closed world: On using machine learning for network intrusion detection}, in 2010 IEEE symposium on security and privacy, IEEE, 2010, pp.~305--316.

\bibitem{Stanley}
{\sc R.~P. Stanley}, {\em Combinatorics and commutative algebra}, vol.~41, Springer, 2007.

\bibitem{Cheng}
{\sc J.~Tang, Y.~Yang, W.~Wei, L.~Shi, L.~Su, S.~Cheng, D.~Yin, and C.~Huang}, {\em Graphgpt: Graph instruction tuning for large language models}, arXiv preprint arXiv:2310.13023,  (2023).

\bibitem{tirumala2019survey}
{\sc S.~Tirumala, M.~R. Valluri, and G.~Babu}, {\em A survey on cybersecurity awareness concerns, practices and conceptual measures}, in 2019 International Conference on Computer Communication and Informatics (ICCCI), IEEE, 2019, pp.~1--6.

\bibitem{walters2021cyber}
{\sc R.~Walters and M.~Novak}, {\em Cyber security}, in Cyber security, artificial intelligence, data protection \& the law, Springer, 2021, pp.~21--37.

\end{thebibliography}
\bibliographystyle{siam}

\end{document}